\documentclass{amsart}

\usepackage{amstext,amsmath,amssymb,amsfonts,amsthm}
\usepackage{graphicx, color}
\usepackage{mathtools}
\usepackage{array}
\usepackage{paralist}
\usepackage{hyperref}
\usepackage{cleveref}
\usepackage{ctable}
\usepackage{paralist}
\usepackage{fullpage}

\newtheorem{theorem}{Theorem}

\newtheorem{corollary}[theorem]{Corollary}
\newtheorem{proposition}[theorem]{Proposition}
\theoremstyle{remark}
\newtheorem{remark}[theorem]{Remark}

\newcommand{\Pcal}{\mathcal{P}}
\newcommand{\Qcal}{\mathcal{Q}}
\newcommand{\Ppriority}{\mathcal{P}_\textsf{priority}}
\newcommand{\Pconflict}{\mathcal{P}_\textsf{conflict}}
\newcommand{\Qsync}{\mathcal{Q}_\textsf{sync}}
\newcommand{\cadlag}{{c\`adl\`ag}}
\newcommand{\out}{^{\textrm{out}}}

\newcommand{\inc}{^{\textrm{in}}}
\newcommand{\Cpim}{C^-_{\pi}}
\newcommand{\Cpistar}{C^-_{\pi^*}}
\newcommand{\transpose}{^\mathsf{T}}
\newcommand{\R}{\mathbb{R}}
\newcommand{\Rplus}{\mathbb{R}_{\geq 0}}
\newcommand{\Rsplus}{\mathbb{R}_{> 0}}
\newcommand{\N}{\mathbb{N}}

\newcommand{\ie}{\textit{i.e.}}

\begin{document}

\title[Stationary solutions of discrete and continuous PN with priorities]{Stationary solutions of discrete and continuous Petri nets with priorities}
\thanks{The three authors were partially supported by the programme ``Concepts, Syst{\`e}mes et Outils pour la S\'ecurit\'e Globale'' of the French National Agency of Research (ANR), project ``DEMOCRITE'', number ANR-13-SECU-0007-01, and by the PGMO program of EDF and Fondation Math\'ematique Jacques Hadamard.
The first and last authors were partially supported by the programme ``Ing\'enierie Num\'erique \& S\'ecurit\'e'' of ANR, project ``MALTHY'', number ANR-13-INSE-0003.}

\author[X.Allamigeon]{Xavier Allamigeon}
\author[V.B\oe{}uf]{Vianney B\oe{}uf}
\author[S.Gaubert]{St\'ephane Gaubert}
\address[X.~Allamigeon, S.~Gaubert]{INRIA and CMAP, \'Ecole polytechnique, CNRS, Universit\'e de Paris-Saclay}
\address[V.~B\oe{}uf]{\'Ecole des Ponts ParisTech \and INRIA and CMAP, \'Ecole polytechnique, CNRS, Universit\'e de Paris-Saclay \and Brigade de sapeurs-pompiers de Paris}
\email{FirstName.LastName@inria.fr}

\date{\today}

\begin{abstract}
We study a continuous dynamics for a class of Petri nets which allows
the routing at non-free choice places to be determined by priorities rules.
We show that this dynamics can be written in terms of policies which identify the bottleneck places. 
We characterize the stationary solutions, and show that they coincide with the stationary solutions of the discrete dynamics of this class of Petri nets. 
We provide numerical experiments on a case study of an emergency call center, indicating that pathologies of discrete models 
(oscillations around a limit different from the stationary limit) 
vanish by passing to continuous Petri nets.
\end{abstract}

\maketitle

\section{Introduction}
\subsection*{Context}
The study of continuous analogues of Petri nets dates back to the works of David and Alla~\cite{david1987continuous} and Silva and Colom~\cite{silva1987structural} in 1987. It has given rise to a large scope of research in the field of Petri nets.
Whereas classical (discrete) Petri nets belong to the class of discrete event dynamic systems,
the circulation of tokens in continuous Petri nets is a continuous phenomenon: 
tokens are assumed to be fluid, \ie, a transition can fire an infinitesimal quantity of tokens. In this way, the continuous dynamics can be represented by a system of ordinary differential equations or differential inclusions.

Continuous Petri nets are usually introduced as a relaxed approximation of Petri nets, that helps understanding some of the properties of the underlying discrete model, 
allowing one to overcome the state space explosion that can occur in the latter.
The continuous framework can also be seen as a scaling limit of a class of stochastic Petri nets (see~\cite{darling2008diff}), 
where the marking $M_p$ of place $p$ in the fluid model is the finite limit of $M_p(N)/N$, with $N$ being a scaling ratio tending to infinity, and where the firing times of transitions follow a Poisson distribution.

An important effort has been devoted to the comparison between continuous nets and their discrete counterparts.
For example, the relationship between reachability of continuous Petri nets and  of discrete Petri nets is well understood (see~\cite{recalde1999autonomous}).
A recent introduction to continuous models can be found in \cite{vazquez2013introduction}, while a more extensive reference is \cite{david2010discrete}.

In order to evaluate the long-term performance of Petri nets,
one has to characterize the stationary or steady states of the Petri nets dynamics. 
Cohen, Gaubert and Quadrat~\cite{cohen1995asymptotic} introduced an approximation of a discrete Petri net by a fluid, piecewise affine dynamics with finite delays, and showed that the limit throughput does exist  
for a class of consistent and free choice Petri nets.
In the more recent work of Gaujal and Giua~\cite{gaujal2004optimal}, the result is extended to larger classes of Petri nets, and the stationary throughputs are computed as the solutions of a linear program. The results obtained using this fluid approximation hardly apply to the discrete model,
up to a remarkable exception identified by Bouillard, Gaujal and Mairesse~\cite{bouillard2006extremal} (bounded Petri nets under total allocation).
This reference illustrates the many difficulties that arise from the discrete setting (\emph{e.g.}, some firing sequences may lead to a deadlock).

In the continuous dynamics setting, with time attached to transitions, 
Recalde and Silva~\cite{recalde2000pn} showed that 
the steady states of free choice Petri nets as well as
upper bounds of the throughputs in larger classes of Petri nets can be determined by linear programming.
However, in general, the asymptotic throughputs are non-monotone with respect to the initial marking or the firing rates of the transitions~\cite{mahulea2006performance}.
An example of oscillations in infinite time around a steady state is also given in~\cite{mahulea2008steady}. 

\subsection*{Contributions}
We propose a continuous dynamics of Petri nets where time is attached to places and not to transitions. The main novelty is that it handles a class of Petri nets in which tokens can be routed according to priority rules (Section~\ref{sec:semantics}).
We initially studied this class in~\cite{allamigeon2015performance} in the discrete setting, motivated by an application to the performance analysis of an emergency call center.

We show that the continuous dynamics can equivalently be expressed in terms of {\em policies}. 
A policy is a map associating with every transition one of its upstream places.
In this way, the dynamics of the Petri net can be written as an infimum of the dynamics of subnets induced by the different policies. 
The policies reaching the infimum indicate the places which are bottleneck in the Petri net. 
On any time interval in which a fixed policy reaches the infimum, the dynamics reduces to a linear dynamics (Section~\ref{sec:policies}).

We characterize the stationary solutions in terms of the policies of the Petri net.
This allows us to set up a correspondence between the (ultimately affine) stationary solutions 
of the discrete dynamics that were described in~\cite{allamigeon2015performance} and the stationary solutions of the continuous dynamics (Section~\ref{sec:4}). 
We also relate the continuous stationary solutions to the initial marking of the Petri net.
This relies on restrictive assumptions, in particular the semi-simplicity of a 0 eigenvalue of a matrix associated with a policy.

We finally provide some numerical simulations of the continuous dynamics. 
We consider a model of emergency call center with two hierarchical levels for handling calls,
originating from a real case study (17-18-112 call center in the Paris area)~\cite{allamigeon2015performance}.
On this Petri net,
numerical experiments illustrate the convergence of the trajectory towards the stationary solution.
This exhibits an advantage of the continuous setting in comparison to the discrete one, 
in which, for certain values of the parameters, the asymptotic throughputs computed by simulations differ from the stationary solutions (Section~\ref{sec:numerical experiments}).

\subsection*{Related work}
The motivation of this work stems from our previous study~\cite{allamigeon2015performance}, in which we addressed the same class of Petri nets with priorities in the discrete setting, 
and applied it to the performance analysis of an emergency call center.
The discrete dynamics is shown there to be given by piecewise affine equations (tropical analogues of rational equations).
The idea of modeling priority rules by piecewise affine dynamics
originated from Farhi, Goursat and Quadrat~\cite{farhi}, who applied it to a special class of road traffic models.
In the discrete setting, limit time-periodic behaviors can occur. They may lead to asymptotic throughputs different from the affine stationary solutions of the dynamics,
a pathology which motivates our study of a continuous version of the dynamics.

The ``continuization'' of our dynamics draws inspiration from the original continuous model where time is attached to transitions.
In particular, the situation in which the routing of a token at a given place is influenced by the firing times of the output transitions through a race policy has received much attention, see~\cite{vazquez2013introduction}.
Here, we address the situation in which
the routing is specified by priority or preselection rules
which are independent of the processing rates. To do so, it
is convenient to attach times to places, instead of attaching firing
rates to transitions. We point out in Remark~\ref{rk:comparison}  that our model can be reduced to a variant of the standard continuous model~\cite{vazquez2013introduction} in which we allow immediate transitions and require
non-trivial routings to occur only at these transitions.  
A benefit of our presentation is to allow a more transparent
comparison between the continuous model and
the discrete time piecewise affine models studied in~\cite{cohen1995asymptotic,gaujal2004optimal,allamigeon2015performance}. 

The use of the term ``policy'' refers to the theory of Markov decision processes, owing to the analogy between the discrete time dynamics and the value function of a semi-Markovian decision process. 
Note that in the context of continuous Petri nets, policies are also known as ``configurations'', see~\cite{mahulea2006performance} for an example.

\section{Continuous dynamics of Petri nets}
\label{sec:semantics}

\subsection{General notation} \label{sec:notations}
A Petri net consists of a set $\Pcal$ of places, a set $\Qcal$ of transitions and a set of arcs $\mathcal{E} \subset (\Pcal \times \Qcal) \cup (\Qcal \times \Pcal)$. 
Every arc is given a valuation in $\N$.
Each place $p \in \Pcal$ is given an initial marking $M^0_p \in \N$, which represents the number of tokens  initially occurring in the place. 

We denote by $a^+_{qp}$ the valuation of the arc from transition $q$ to place $p$, with the convention that $a^+_{qp}=0$ if there is no such arc. Similarly, 
we denote by $a^-_{qp}$ the valuation of the arc from place $p$ to transition $q$, with the same convention. We set $a_{qp} := a^+_{qp} - a^-_{qp}$.
The $\Qcal\times \Pcal$ matrix $A = (a_{qp})_{q \in \Qcal, p \in \Pcal}$ is 
referred to as the \emph{incidence matrix} of the Petri net,
and its transpose matrix $C := A\transpose$ as its \emph{token flow matrix}.
We also denote by $C^+$ (resp.\ $C^-$) the $\Pcal \times \Qcal$ matrix with entry $a_{qp}^+$ (resp.\ $a_{qp}^-$), so that $C = C^+ - C^-$.
We limit our attention to \emph{pure} Petri nets, \ie, Petri nets with no self-loop: for every pair $(q,p)$, at least one of $a^+_{qp}$ and $a^-_{qp}$ is zero.

We denote by $q\inc$ the set of upstream places of transition $q$ and by $q\out$ the set of downstream places of transition $q$. Similarly, we use the notation $p\inc$ and $p\out$ to refer to the sets of input and output transitions of a place $p$.

\subsection{Petri nets with free choice and priority routing}

In this paper, we consider a class of Petri nets in which places are either free choice or subject to priority. Recall that a place $p\in \Pcal$ 
is said to be \emph{free choice}
if either all the output transitions $q \in p\out$ satisfy $q\inc = \{p\}$ (\emph{conflict}, see Figure~\ref{fig:configuration}(a)), or  $|p\out| = 1$ (\emph{synchronization}, see Figure~\ref{fig:configuration}(b)). A place is \emph{subject to priority} if its tokens are routed to output transitions according to a priority rule. We refer to Figure~\ref{fig:configuration}(c) for an illustration. For the sake of simplicity, we assume that each place subject to priority has exactly two output transitions, and that any transition has at most one upstream place subject to priority. Given a place $p$ subject to priority, we denote by $q^+(p)$ and $q^-(p)$ its two output transitions, with the convention that $q^+(p)$ has priority over $q^-(p)$. For the sake of readability, we use the notation $q^+$ and $q^-$ when the place $p$ is clear from context.

The set of transitions such that every upstream place $p$ satisfies $|p\out| = 1$ is referred to as $\Qsync$ and the set of free choice places that have at least two output transitions is referred to as $\Pconflict$. 
We denote by $\Ppriority$ the set of places subject to priority.
The sets $(\Pconflict)\out$, $\Qsync$ and $(\Ppriority)\out$ form a partition of $\Qcal$.
Figure~\ref{fig:configuration} hence summarizes the three possible place/transition patterns that can occur in this class of Petri nets.
\begin{figure}
\begin{center}
\includegraphics{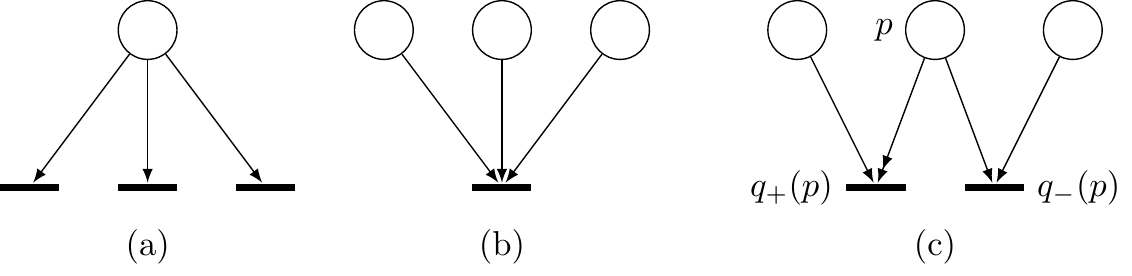}
\caption{Conflict, synchronization and priority patterns}\label{fig:configuration}
\end{center}
\end{figure}

\subsection{Continuous dynamics and routing rules}

We now equip the Petri net with a continuous semantics. Given a transition $q$, we associate a flow $f_q(t)$ which represents
the instantaneous firing rate of transition $q$ at time $t$.
We also associate with each place $p$ a {\em marking} $M_p(t)$, which is
a continuous real valued function of the time $t$.
In the case of {\em discrete} timed Petri nets,
one typically requires that every token stays a minimum time in the place,
--- at this stage, the token may be considered as {\em under processing} ---
before becoming available for the firing of output transitions. 
To capture this property in the continuous setting,
we assume that the marking $M_p(t)$ can be decomposed
as $M_p(t)=m_p(t)+w_p(t)$, where $m_p(t)$ is the quantity
of tokens under processing and $w_p(t)$ is the quantity of tokens
waiting to contribute to the firing of an output transition.

We associate with each place $p$ a time constant $\tau_p>0$.
 Each token entering in a place is processed 
with the rate $1/\tau_p$.
This leads to the following differential equation:
\begin{equation} \label{eq:mp}
  \dot{m}_p(t) = \sum_{q \in p\inc}a_{qp}^+f_q(t) - \dfrac{m_p(t)}{\tau_p} \, .
\end{equation}
The evolution of the number of tokens waiting in place $p$ is 
described by the relation:
\begin{equation} \label{eq:wp}
  \dot{w}_p(t) = \dfrac{m_p(t)}{\tau_p} - \sum_{q\in p\out} a_{qp}^- f_{q}(t) \, .
\end{equation}
Moreover, for all transition $q$, we require that
\begin{equation}
\min_{p\in q\inc, \, w_p(t)=0}\Big(
\dfrac{m_p(t)}{\tau_p} - \sum_{q'\in p\out} a^-_{q'p}f_{q'}(t)\Big)  =0 
\,.
\label{eq:earliest}
\end{equation}
In particular, this implies that at least one place $p \in q\inc$ verifies $w_p(t) = 0$. In this case, \eqref{eq:earliest} means that each of the upstream places $p$ that has a zero quantity of waiting tokens ($w_p(t)=0$)
must satisfy $\dot{w}_p(t)\geq 0$, and that at least one
of these places satisfies $\dot{w}_p(t)=0$. 
In other words, there is at least one 
\emph{bottleneck} upstream place $p$ of $q$, which has no waiting tokens 
and whose outgoing
flow $\sum_{q'\in p\out} a^-_{q'p}f_{q'}(t)$ coincides
with its processing flow ${m_p(t)}/{\tau_p}$.

The relation provided in~\eqref{eq:earliest} can be simplified in the case of conflict and synchronization patterns. In more detail, if $q$ has a unique upstream place $p$, and this place is free choice ({conflict}), then~\eqref{eq:earliest} reduces to:
\begin{equation}\label{eq:conflict}
\dfrac{m_p(t)}{\tau_p} - \sum_{q'\in p\out} a^-_{q'p}f_{q'}(t) =0 \,.
\end{equation}
Now, if $q$ has several upstream places, which are all free choice ({synchronization}), then \eqref{eq:earliest} reads as: 
\begin{equation}\label{eq:sync}
f_{q}(t) = \min_{p\in q\inc, \, w_p(t)=0} \dfrac{m_p(t)}{a^-_{qp}\tau_p} \, .
\end{equation}
This equation also holds if $|q\inc| = 1$ and if the upstream place of $q$ has a single output transition.

We respectively denote by $m(t)$, $w(t)$ and $f(t)$ the vectors of entries $m_p(t)$, $w_p(t)$ and $f_q(t)$. 

Albeit the dynamics that we presented so far is piecewise affine, a trajectory $t \mapsto (m(t), w(t), f(t))$ may be discontinuous. Indeed, in~\eqref{eq:sync}, the set of the places over which the minimum is taken may change over time. If at time $t$, there is a new place $p \in q\inc$ such that $w_p(t)$ cancels, and if the quantity ${m_p(t)}/(a^-_{qp}\tau_p)$ is sufficiently small, then the minimum in~\eqref{eq:sync} (and subsequently the flow $f_q(t)$) discontinuously jumps to the latter value. 

Initial conditions of the dynamics are specified by a pair $(m(t_i),w(t_i))$ such that the minimum in~\eqref{eq:earliest} makes sense, 
\ie, at least one $w_p(t_i)$ is equal to $0$ for each set of places $q\inc$. 
One can easily show that if the set $\{ p \in q\inc \colon w_p(t) = 0\} $ is nonempty for all transition $q \in \Qcal$ at time $t = t_i$, then it remains nonempty for all time $t \geq t_i$.

The dynamics~\eqref{eq:mp}--\eqref{eq:earliest} may admit different
trajectories for a given initial condition. These correspond
to different routings of tokens in
places with several output transitions. However, each of these
trajectories satisfies the conservation law:
\begin{equation} \label{eq:flow_conservation_eq}
  \dot{m}(t) + \dot{w}(t) = C f(t) \, ,
\end{equation}
where $C$ is the token flow matrix of the Petri net.
Recall that a \emph{P-invariant} of the Petri net refers to a nonnegative solution $y$ of the system $y\transpose C = 0$. In the discrete setting, a P-invariant corresponds to a weighting of places that is constant for any reachable marking, meaning that the quantity $y \transpose M$ is preserved under any firing of transition. An analogous statement holds in the continuous setting:
\begin{proposition}
Given a P-invariant $y$ of the Petri net, the quantity $y\transpose (m(t) + w(t))_{p \in \Pcal}$ is independent of~$t$.

In particular, if the entries of $y$ are all positive, then the Petri net is bounded, \ie, each function $t \mapsto M_p(t)$ is bounded.
\end{proposition}

\begin{proof}
The proof consists in multiplying both sides of~\eqref{eq:flow_conservation_eq} by the invariant $y$.
It follows that the derivative of $y\transpose (m(t) + w(t))_{p \in \Pcal}$ is zero.
If the entries of $y$ are all positive, then $\bigl(y\transpose(m(0) + w(0))_{p \in \Pcal}\bigr)/y_p$ is an upper bound to the marking of any place $p$.
\end{proof}

The following proposition collects several homogeneity properties of the continuous dynamics:

\begin{proposition} \label{thm:prop2}
  Let $(m(t),w(t),f(t))$ be a trajectory solution of the dynamics~\eqref{eq:mp}--\eqref{eq:earliest}, with the initial markings $(m_p(0))_{p \in \Pcal}$, and the holding times $(\tau_p)_{p \in \Pcal}$ and let $\alpha \in \Rsplus$, then:
  \begin{enumerate}[(i)]
  \item $(\alpha m(t), \alpha w(t), \alpha f(t))$ is a trajectory solution of the dynamics, associated with the initial markings $(\alpha m_p(0))_{p \in \Pcal}$.
  \item $(m(t/\alpha),w(t/\alpha),(1/\alpha)f(t/\alpha))$ is a trajectory solution of the dynamics, associated
  with the holding times $(\alpha \tau_p)_{p \in \Pcal}$ and the same initial conditions.
  \item let $x$ be a vector of the kernel of $C$, and $D = \text{diag}(\tau)$ be the $\Pcal \times \Pcal$ diagonal matrix such that $D_{pp} = \tau_p$, then $(m(t) + \alpha D C^+ x, w(t), f(t) + \alpha x)$ is a trajectory solution of the dynamics, associated with the initial markings $(m(0) + \alpha D C^+ x)$.
  \end{enumerate}
\end{proposition}
\begin{proof}
The first two statements derive easily from the homogeneity properties of Equations~\eqref{eq:mp}--\eqref{eq:earliest}.
For the third statement, one can note that adding $\alpha x_q$ to the $f_q(t)$ and adding $\alpha \sum_{q\in p\inc} a_{qp}^+ x_q$ to the $m_p(t)/\tau_p$ in~\eqref{eq:mp}--\eqref{eq:earliest} does not change the right hand sides of~\eqref{eq:mp} and~\eqref{eq:wp}, or the expression within the minimum in~\eqref{eq:earliest}.
For~\eqref{eq:wp} and~\eqref{eq:earliest}, this is due to the fact that $(C^+ - C^-) x = C x = 0$.
\end{proof}
We now complete the description of the continuous dynamics by additional equations which arise from the specification of routing rules. Such rules occur in the following two situations:
\begin{asparadesc}
\item[Conflict.] Given $p \in \Pconflict$, we suppose that tokens are routed according to a stationary distribution specified  by weights $\mu_{qp} > 0$ associated with each output transition $q$. Therefore,
\begin{equation}\label{eq:conflict_routing}
  \forall p \in \Pconflict, \; \forall q \in p\out, \quad a^-_{qp}f_q(t) = \mu_{qp} \dfrac{m_p(t)}{\tau_p} \, .
\end{equation}

\item[Priority.] Let $p \in \Ppriority$, and $q_+$ and $q_-$ be the two output transitions, as illustrated in \Cref{fig:configuration}(c). In order to specify that the flow is routed in priority to transition $q_+$, we require that:
\begin{align}
f_{q_+}(t) &= \min_{r \in q_+\inc, \, w_r(t) = 0} \dfrac{m_r(t)}{a^-_{q_+{r}}\tau_r} \,,
\label{eq:p1}\\
f_{q_-}(t) &= \begin{cases}
    \min_{r \in q_-\inc \setminus \{{p}\},  w_r(t) = 0} \dfrac{m_r(t)}{a^-_{q_-r}\tau_r} & \text{if}\  w_{{p}}(t) \neq 0 \, ,\\
\min \Bigl(\dfrac{m_{{p}}(t)}{a^-_{q{p}}\tau_{{p}}} - \dfrac{a^-_{q_+{p}}}{a^-_{q{p}}}f_{q_+}(t) \, , 
\min_{r \in q_-\inc \setminus \{{p}\}, w_r(t) = 0} \dfrac{m_r(t)}{a^-_{q_-r}\tau_r}
\Big) 
& \text{if}\  w_{{p}}(t) = 0 \, .
    \end{cases} 
\label{eq:p2}
\end{align}
\end{asparadesc}
The expression of $f_{q_-}(t)$ in~\eqref{eq:p2}, when $w_{p}=0$,
indicates that only the outgoing flow from ${p}$
that is not already consumed by the priority transition $q_+$ is
available to $q_-$.

The first two properties of homogeneity in Proposition~\ref{thm:prop2} are still satisfied by the dynamics extended by the 
routing rules~\eqref{eq:conflict_routing}--\eqref{eq:p2}.

\begin{remark}\label{rk:comparison}
We already mentioned in the introduction that our model differs from the standard continuous Petri net model in which transitions are equipped with firing rates, in the sense that in the latter model, the flows of the output transitions of a given place are pairwise independent.
To overcome this limitation, 
\emph{immediate transitions} have been introduced~\cite{recalde2006improving}. 
These transitions come with the specification of routing rules, for instance, in the case of conflict pattern. 
In this way, our model could be reduced to a classical continuous model enriched with immediate transitions.
In this reduction, we require timed transitions to have exactly one upstream place and one downstream place, 
so that all the routing is determined by immediate transitions, which inherit the equations defined in our place-timed dynamics.

Simply put, our model is the continuous analogue
of discrete Petri nets equipped with ``holding durations'', in which tokens are frozen during processing, whereas the usual continuous Petri net model can be seen as the continuous analogues of Petri nets with ``enabling durations'', in which transitions preempt tokens. We refer to~\cite{bowden2000brief} for a discussion on the meaning of time in Petri nets.
\end{remark}

\section{Policies and bottleneck places} \label{sec:policies}
The analysis of the piecewise affine dynamical system~\eqref{eq:mp}--\eqref{eq:earliest} leads
to introduce the notion of {\em policy}.
Fixing a policy allows one to solve the dynamics on a region where it is linear.
We shall see in \Cref{sec:4}
that policies also arise in the characterization
of stationary solutions. 

Even if our continuous dynamics holds for more general classes of Petri nets, 
we focus in the following on strongly connected, autonomous Petri nets, so that each transition has at least one upstream place.

We observe that the dynamics of Petri nets with free choice and priority routing~\eqref{eq:mp}--\eqref{eq:wp}, \eqref{eq:sync} and~\eqref{eq:conflict_routing}--\eqref{eq:p2}
is linear on each region where the arguments of the minimum operators do not change.
More precisely, at any time $t$, for any transition $q \in \Qcal$, there exists a place $p \in q\inc$ such that $w_p(t)=0$ and $p$ is the unique upstream place of $q$ or $p$ realizes the minimum in the expression \eqref{eq:sync}, \eqref{eq:p1} or \eqref{eq:p2} of $f_q(t)$.
Place $p$ is then referred to as the \emph{bottleneck place} of transition $q$ at time $t$.

We define a \emph{policy} $\pi$
as a function from $\Qcal$ to $\Pcal$, which maps any transition $q$ to one of its upstream places $p_{\pi}(q) \in q\inc$.
A policy is meant to indicate the bottleneck place of each transition $q$.
We denote by $S_{\pi}$ the {\em selection matrix} associated with~$\pi$, that is, the $\Qcal \times \Pcal$ matrix such that $(S_{\pi})_{qp} = 1$ if $p = p_{\pi}(q)$, and $0$ otherwise. In particular, $(S_{\pi})_{qp} = 1$ implies that $a_{qp} < 0$.

Note that, if $p$ realizes the minimum in one of the equations~\eqref{eq:sync}, \eqref{eq:p1} or~\eqref{eq:p2} for some transition, then $p$ also realizes the minimum in~\eqref{eq:earliest}.
The converse is not true if places are subject to priority.
For $p$ denoting a priority place and $q_+$ its priority output transition,
if $p$ realizes the minimum in~\eqref{eq:earliest} for transition $q_+$, then, $p$ does not necessarily realize the minimum in~\eqref{eq:p1}.
In other words, our definition of a bottleneck place is dependent on the routing rules of the net.

We point out that notions comparable to policies are used in~\cite{mahulea2006performance} in the context of continuous Petri nets with time attached to transitions.

\vspace{1em}
The dynamics of a Petri net can be expressed in terms of the different policies of the net:
at any time $t$, there is a policy $\pi^*$ (we can note $\pi^*(t)$ if we want to emphasize the dependence on time) such that
\begin{equation*}
  \forall q \in \Qcal \, , \quad w_{p_{\pi^*}(q)}(t) = 0 \,, 
\end{equation*}
and
\begin{equation*}
\begin{aligned}
  &\forall q \in \Qcal \text{ s.t.\ } p_{\pi^*}(q) \not\in (\Pconflict \cup \Ppriority) \,, \; \dfrac{m_{p_{\pi^*}(q)}(t)}{\tau_{p_{\pi^*}(q)}} = a^-_{qp_{\pi^*}(q)} f_q(t) \,,\\
  &\forall q \in \Qcal \text{ s.t.\ } p_{\pi^*}(q) \in \Pconflict \,, \quad \dfrac{m_{p_{\pi^*}(q)}(t)}{\tau_{p_{\pi^*}(q)}} = \dfrac{a^-_{qp_{\pi^*}(q)}}{\mu_{qp_{\pi^*}(q)}} f_q(t) \,,\\
  &\forall q_+ \in \Qcal \text{ s.t.\ } p_{\pi^*}(q_+) \in \Ppriority \,, \; \dfrac{m_{p_{\pi^*}(q_+)}(t)}{\tau_{p_{\pi^*}(q_+)}} = a^-_{q_+p_{\pi^*}(q_+)} f_{q_+}(t) \,,\\
  &\forall q_- \in \Qcal \text{ s.t.\ } p_{\pi^*}(q_-) \in \Ppriority \,, \; \dfrac{m_{p_{\pi^*}(q_-)}(t)}{\tau_{p_{\pi^*}(q_-)}}
   = a^-_{q_-p_{\pi^*}(q_-)} f_{q_-}(t)
   + a^-_{q_+p_{\pi^*}(q_+)} f_{q_+}(t) \,.
\end{aligned}
\end{equation*}

Now, for any policy $\pi$, we denote by $\Cpim$ the $\Qcal \times \Qcal$ matrix such that the right-hand side of this system of equations reads $\Cpim f(t)$
, where 
$f(t)$ is the vector of the $(f_q(t))_{q \in \Qcal}$.
In particular, the above system writes 
\[S_{\pi^*} \left(\dfrac{m(t)}{\tau}\right) = \Cpistar f(t) \,, \]
where $(m(t)/\tau)$ is the vector of the $(m_p(t)/\tau_p)_{p \in \Pcal}$.
The diagonal entries of $\Cpim$ are positive.
Moreover, if we order each transition of type $q_+$ before its associated transition $q_-$, the matrix $\Cpim$ becomes lower triangular.\footnote{We recall that, in our class of Petri nets, we assume that each transition has at most one upstream place subject to priority,
so that this re-ordering is valid.}
Hence, $\Cpim$ is invertible.
Matrix $\Cpim$ can be seen as a specification of the downstream token flow matrix of the Petri net $C^-$ (introduced in Section~\ref{sec:notations}), associated with the policy $\pi$.

With this notation, the continuous dynamics of Petri nets with free choice and priority routing reads:
\begin{subequations}\label{eq:matrix_dyn_system}
\begin{align}
  f(t) &= \inf_{\pi \text{ s.t. } S_{\pi} w (t) = 0} (\Cpim)^{-1} S_{\pi} \left( \dfrac{m(t)}{\tau} \right) \label{eq:matrix_dyn_f} \,,\\
  \dot{m}(t) &= C^+ f(t) - \dfrac{m(t)}{\tau} \label{eq:matrix_dyn_dot_m} \,,\\
  \dot{w}(t) &= \dfrac{m(t)}{\tau} - C^- f(t) \label{eq:matrix_dyn_dot_w} \,,
  \end{align} 
\end{subequations}%
where the infimum must be understood for the partial order over $\R^\Qcal$ induced by $\leq$. 
Note that there is at least one policy $\pi^*$ (depending on $t$) attaining the infimum.
It suffices to choose the policy $\pi^*$ introduced earlier
(\ie, to choose a policy that attains the minimum componentwise).

By choosing an upstream place for each transition, a policy actually defines a candidate ``bottleneck net'' of the Petri net, that is, a subnet with all the transitions of the original Petri net, and such that each transition has a unique upstream place.
On each of these subnets, the dynamics is linear and yields a unique trajectory for a given initial condition.
The trajectory is solved on a subnet of the original Petri net, but one can easily recover the solution over the whole Petri net.
This applies to the original dynamics of the system, on any time interval over which the infimum is reached by a constant policy,
as
stated in the following proposition.

\begin{proposition} \label{thm:fc_prio_unicity}
Suppose that there is a policy $\pi^*$ which reaches the infimum in~\eqref{eq:matrix_dyn_f} for all time $t$ in the interval $[t_i, t_f]$. Then the dynamics of the Petri net with free choice and priority routing reduces to a linear system, which admits a unique solution, given the initial conditions $(m(t_i), w(t_i))$.
\end{proposition}

\begin{proof}
If $\pi^*$ reaches the infimum for any $t \in [t_i, t_f]$, then the continuous dynamics of the Petri net reads:
\begin{subequations}\label{eq:const_pi_dyn_system}
\begin{align}
  \Cpistar f(t) &= S_{\pi^*} \left(\dfrac{m(t)}{\tau} \right) \label{eq:const_pi_dyn_f} \,, \\
  \dot{m}(t) &= C^+ f(t) - \dfrac{m(t)}{\tau} \label{eq:const_pi_dyn_dot_m} \,, \\
  \dot{w}(t) &= \dfrac{m(t)}{\tau} - C^- f(t) \label{eq:const_pi_dyn_dot_w} \,, \\
  S_{\pi^*} w (t) &= 0  \,, \label{eq:const_pi_dyn_w}
  \end{align} 
\end{subequations}
which is a linear system.

We multiply~\eqref{eq:const_pi_dyn_dot_m} by $S_{\pi^*}$,
and replace the term $S_{\pi^*} (m(t)/\tau)$ by its expression given in~\eqref{eq:const_pi_dyn_f}.
This leads to:
\begin{equation*}
  S_{\pi^*} \dot{m}(t) = S_{\pi^*}C^+ f(t) - \Cpistar f(t) \,.
\end{equation*}

Let $D = \text{diag}(\tau)$ be the $\Pcal \times \Pcal$ diagonal matrix such that $D_{pp} = \tau_p$,
then $D_{\pi} \mathrel{:=} S_{\pi} D {S_{\pi}}\transpose$ is the $\Qcal \times \Qcal$ diagonal matrix such that $(D_{\pi})_{qq} = \tau_{p_{\pi}(q)}$.
Equation~\eqref{eq:const_pi_dyn_f} then writes $S_{\pi^*} m(t) = D_{\pi^*} \Cpistar f(t)$.
This leads to:
\begin{equation*}
  \dot{f}(t) = (\Cpistar)^{-1} D_{\pi^*}^{-1} \left( S_{\pi^*} C^+  -  \Cpistar \right) f(t) \, ,
\end{equation*}
which is an ordinary differential system.
Moreover, the $f(t_i)$ can be obtained from the $m(t_i)$ by~\eqref{eq:const_pi_dyn_f},
so that this system admits a unique solution $f$ for all $t \in [t_i, t_f]$.

Given this solution $f$, one can successively solve the differential system in ${m}$ given by~\eqref{eq:const_pi_dyn_dot_m} and the differential system in ${w}$ given by~\eqref{eq:const_pi_dyn_dot_w}, whose initial conditions are known,
so that the whole dynamics admits a unique trajectory.
\end{proof}

\section{Stationary solutions}\label{sec:4}
In this section, we prove that the stationary solutions of the continuous and discrete dynamics of a timed Petri net with free-choice and priority routing are the same. 
To do so, we first
recall in Section~\ref{sec:stationary_discrete} the formulation of the discrete dynamics and the associated stationary solutions given in~\cite{allamigeon2015performance}.

\subsection{Stationary solutions of the discrete dynamics} \label{sec:stationary_discrete}

The discrete dynamics of Petri nets with free choice and priority is expressed in terms of \emph{counter variables} associated with transitions and places.
Given a transition $q$, the counter variable $z_q: \Rplus \to \N$ denotes the number of firings of~$q$ that occurred up to time $t$ included. Similarly, the counter variable of place $p$ is a function $x_p:\Rplus \to \N$ which represents the number of tokens that have visited place $p$ up to time $t$ included (taking into account the initial marking). On top of being non-decreasing, the counter variables are \emph{\cadlag} functions, which means that they are right continuous and have a left limit at any time.

In this setting, the parameter $\tau_p$ associated with the place $p$ represents a minimal holding time. 
It is shown in~\cite{allamigeon2015performance} that, if tokens are supposed to be fired as early as possible, the counter variables satisfy the following equations (we generalize the equations to the case with valuations):
\begin{subequations}\label{eq:discrete}
\begin{align}
 \forall p \in \Pcal \,&, \quad x_p(t) =  M_p^0  + \sum_{q\in p\inc} a^+_{qp} z_q(t) \, , \label{eq:pnpriority1} \\
 \forall p \in \Pconflict \,&, \quad \sum_{q \in p\out} a^-_{qp} z_q(t) = x_p(t - \tau_p) \, ,\label{eq:pnpriority2} \\
 \forall q \in \Qsync \,&, \quad z_q(t) = \min_{p \in q\inc} x_p(t - \tau_p) / a^-_{qp} \, ,\label{eq:pnpriority3}
 \end{align}
\begin{align}
\forall & p \in \Ppriority \,, \nonumber\\
 &
 z_{q_+}(t) = \min\biggl( 
\Bigl(\dfrac{1}{a^-_{q_+p}}x_p(t-\tau_p) - \dfrac{a^-_{q_-p}}{a^-_{q_+p}} \lim_{s \uparrow t} z_{q_-}(s) \Bigr),
\min_{{r \in q_+\inc, r \neq p}} \dfrac{1}{a^-_{q_+r}}x_r(t - \tau_r) \biggr)\, ,
 \label{eq:pnpriority4} \\
&
z_{q_-}(t) = \min \biggr(
\Bigl(\dfrac{1}{a^-_{q_-p}}x_p(t-\tau_p) - \dfrac{a^-_{q_+p}}{a^-_{q_-p}} z_{q_+}(t)\Bigr),
\min_{{r \in q_-\inc, r \neq p}}\dfrac{1}{a^-_{q_-r}} x_r(t - \tau_r) \biggr) \, ,
\label{eq:pnpriority5}
\end{align}
\end{subequations}
where $q_+$ ($q_-$) is the priority (non priority) output transition of $p \in \Ppriority$.

Note that if all the holding times $\tau_p$ are integer multiples of a fixed time $\delta$,  the left limit $\lim_{s \uparrow t}z_{q_-}(s)$ in~\eqref{eq:pnpriority4} can be replaced
by $z_{q_-}(t-\delta)$. This is helpful in particular to simulate
these equations.

In the setting of~\cite{allamigeon2015performance}, all conflicts
are solved by a stationary distribution routing.
The equivalent of the routing rule introduced to solve conflicts in the continuous setting is obtained here by 
allowing the tokens to be shared in fractions,
so that the counter functions take real values.
This corresponds to a \emph{fluid approximation}
of the discrete dynamics. In this setting, 
for each $p \in \Pconflict$ and $q\in p\out$,  we fix $\mu_{qp}>0$, giving
the proportion of the tokens routed from $p$ to $q$. We have:
\begin{equation}
\forall p \in \Pconflict \, , \, \forall q \in p\out \,, \quad z_q(t) = \dfrac{\mu_{qp}}{a^-_{qp}} x_p(t - \tau_p) \ . 
\label{eq:pnpriority6fluid}
\end{equation}

The stationary solutions of the discrete dynamics are defined as functions $x_p$ and $z_q$ satisfying the relations~\eqref{eq:discrete}--\eqref{eq:pnpriority6fluid} and which ultimately behave as affine functions, \ie, $x_p(t) =  u_p + t \rho_p$ and $z_q(t) = u_q + t \rho_q$ for all $t$ large enough. In this case, $\rho_p$ (resp.\ $\rho_q$) represents the asymptotic throughput of place $p$ (resp.\ transition $q$). We have shown in~\cite[Theorem~3]{allamigeon2015performance} that these stationary solutions are precisely given by following system (we generalize the equations to the case with valuations):

\begin{subequations}\label{eq:fixpoint_germ_rho}
\allowdisplaybreaks
\begin{align}
\forall p \in \Pcal \, , \qquad \rho_p& = \sum_{q \in p\inc} a^+_{qp} \rho_q \label{eq:fixpoint_germ1_rho} \,, \\
\forall p \in \Pconflict \, , \forall q \in p\out \, , \qquad \rho_q & = \mu_{qp} \rho_p / a^-_{qp}
\label{eq:fixpoint_germ2_rho} \,, \\
\forall q \in \Qsync \, , \qquad \rho_q & = \min_{p \in q\inc} \rho_p/a^-_{qp} \label{eq:fixpoint_germ3_rho}\,, \\
\forall p \in \Ppriority \, , \qquad
\rho_{q_+} &= \min_{r \in q_+\inc} \rho_r / a^-_{q_+r} \label{eq:fixpoint_germ4_rho}\,, \\
\forall p \in \Ppriority \, , \quad
\rho_{q_-} = \min \Big(
&{\big(\rho_p - a^-_{q_+p}\rho_{q_+}\big)}/{a^-_{q_-p} }, 
\min_{r \in q_-\inc \setminus \{ p \}} {\rho_r}/{a^-_{q_-r}} \Big) \,,
\label{eq:fixpoint_germ5_rho} 
\end{align}
\end{subequations}

\begin{subequations}\label{eq:fixpoint_germ_u}
\allowdisplaybreaks
\begin{align}
\forall p \in \Pcal \, &, \quad u_p = M_p^0 + \sum_{q \in p\inc} a^+_{qp} u_q \label{eq:fixpoint_germ1_u} \,, \\
\forall p \in \Pconflict \, , \forall q \in p\out &, \; u_q = (\mu_{qp}/a^-_{qp}) (u_p - \rho_p \tau_p) 
\label{eq:fixpoint_germ2_u} \,,
\end{align}
\begin{equation}
\forall q \in \Qsync \, , \quad u_q  = \min_{p \in q\inc, \rho_q = \rho_p} (u_p - \rho_p \tau_p)/a^-_{qp} \label{eq:fixpoint_germ3_u} \,,
\end{equation}
\begin{align}
\forall p &\in \Ppriority \, , \nonumber\\ 
u_{q_+} & = \begin{dcases}
\min
\begin{multlined}[t]
\Big( (u_p - \rho_p \tau_p - a^-_{q_-p}u_{q_-})/a^-_{q_+p}, \\
\min_{r \in q_+\inc \setminus \{p\}, \, \rho_{q_+} = \rho_r} (u_r - \rho_r \tau_r)/a^-_{q_+r} \Big) \end{multlined}
\quad &\text{if } \rho_{q_-} = 0 \, ,
 \\
\min_{r \in q_+\inc \setminus \{p\}, \, \rho_{q_+} = \rho_r} (u_r - \rho_r \tau_r)/a^-_{q_+r} \quad &\text{otherwise,}
\end{dcases}\label{eq:fixpoint_germ4_u}
\\
u_{q_-} & = \begin{dcases}
\min
\begin{multlined}[t]
\Big( (u_p - \rho_p \tau_p - a^-_{q_+p} u_{q_+})/a^-_{q_-p}, \\
\min_{r \in q_-\inc \setminus \{p\}, \, \rho_{q_-} = \rho_r} (u_r - \rho_r \tau_r)/a^-_{q_-r} \Big)
\end{multlined}
\quad &\!\!\! \text{if } \rho_{q_-} + \rho_{q_+} = \rho_p \, ,
 \\
\min_{r \in q_-\inc \setminus \{p\}, \, \rho_{q_-} = \rho_r} (u_r - \rho_r \tau_r)/a^-_{q_-r} \quad &\text{otherwise.}
\end{dcases}\label{eq:fixpoint_germ5_u}
\end{align}
\end{subequations}

The above equations are expressed in a more compact form in~\cite{allamigeon2015performance}, using a semiring of germs of affine functions, which encodes lexicographic minimization operations.

\subsection{Stationary solutions of the continuous time dynamics} \label{sec:stationary_continuous}

In the continuous setting, we define a \emph{stationary solution} as a solution $(m,w,f)$ of the continuous dynamics such that for any place, ${m_p}$ is constant and ${w_p}$ is affine ($\dot{w_p}$ is constant). The following theorem provides a characterization of the stationary solutions.

\begin{theorem}\label{th:1}
A triple $(m,w,f)$ of vectors of resp.\ $|\Pcal|$,  $|\Pcal|$ and  $|\Qcal|$ functions from $\Rplus$ to $\Rplus$, with all the $m_p$ constant and all the $w_p$ affine, is a stationary solution of the continuous dynamics if and only if the following conditions hold:
\begin{subequations}
\begin{align}
  \dfrac{m}{\tau} &= C^+ f
  \, , \label{eq:mpcns} \\
  \dot{w} &= \dfrac{m}{\tau} - C^- f
  \, , \label{eq:wpcns} \\
  C f &\geq 0 \, ,\label{eq:stat_nodes_law} \\
\shortintertext{and there exists a policy $\pi^*$, such that}
\label{eq:wp_annuler}
\forall t \,, \quad S_{\pi^*} w(t) &= 0 \, , \\
\label{eq:right_annuler}
\left( S_{\pi^*} C^+  -  \Cpistar \right) f &= 0 \, .
\end{align}
\end{subequations}
\end{theorem}

Note that the existence of an $f \gneq 0$ that satisfies~\eqref{eq:stat_nodes_law}
provides a simple algebraic necessary condition to the existence of a stationary flow in a Petri net.
This corresponds to the net being {\em partially repetitive} (see~\cite{murata1989petri} for a definition).
\begin{proof}
Equations~\eqref{eq:mpcns} and~\eqref{eq:wpcns} are derived from~\eqref{eq:matrix_dyn_dot_m} and~\eqref{eq:matrix_dyn_dot_w}, with $\dot{m} =0$ for a stationary solution.

In a stationary solution, for any place $p$, $\dot{w}_p$ is constant, so that one cannot have $\dot{w}_p < 0$, 
otherwise this would yield $\lim_{t \rightarrow \infty} w_p(t) = -\infty$.
Therefore, by~\eqref{eq:wpcns}, $(m/\tau) \geq C^- f$, 
and by~\eqref{eq:mpcns}, we can replace $(m/\tau)$ by $C^+ f$, and get~\eqref{eq:stat_nodes_law}.

As the $\dot{w}$ are constant, if, for some place $p$ and at some time $t_0 > 0$, $w_p(t_0) = 0$, then $\dot{w}_p =0$,
(otherwise it would contradict $w_p(t) \geq 0$ for $0 \leq t < t_0$ or for $t > t_0$).
Hence, the set of places $p$ such that $w_p(t) = 0$ is independent of time for $t > 0$.

Moreover, the $m_p$ are constant, so that, if a policy $\pi$ attains the minimum in~\eqref{eq:matrix_dyn_f} at some time, then it attains the minimum at any time.
This means that, if $(m, w, f)$ is a solution of the continuous dynamics, then there exists a policy $\pi^*$ such that:
\begin{align*}
  \forall t \,, \quad S_{\pi^*} w(t) &= 0 \, , \\
  \Cpistar f &= S_{\pi^*} \left( \dfrac{m}{\tau} \right) \,.
\end{align*}
Now, by~\eqref{eq:mpcns} again, we can replace
$m/\tau$ by $C^+ f$ in the above equation, and we get Equations~\eqref{eq:wp_annuler} and~\eqref{eq:right_annuler}.

Conversely, suppose that a triple of functions $(m,w,f)$ satisfies the conditions of the theorem, with policy $\pi^*$. 
We prove that the the relations given in~\eqref{eq:matrix_dyn_system} describing the dynamics are satisfied.
First, \eqref{eq:matrix_dyn_dot_m} and~\eqref{eq:matrix_dyn_dot_w} are derived from~\eqref{eq:mpcns} and~\eqref{eq:wpcns}, with $\dot{m}_p = 0$.
We also note that, in Equations~\eqref{eq:stat_nodes_law} and~\eqref{eq:right_annuler},
replacing the term $C^+ f$ by $m/\tau$  (by~\eqref{eq:mpcns}) leads to the following equations:
\begin{align} 
  C^- f &\leq \dfrac{m}{\tau}  \label{eq:stationary_ineq_mp} \,, \\
  f &= (C_{\pi^*}^-)^{-1} S_{\pi^*} (\dfrac{m}{\tau})\,. \label{eq:right_annuler_mp}
\end{align} 

Equations~\eqref{eq:wp_annuler} and~\eqref{eq:right_annuler_mp} show that $\pi^*$ attains the equality in~\eqref{eq:matrix_dyn_f}.
Hence, in order to prove~\eqref{eq:matrix_dyn_f}, it is sufficient to prove that, for any $\pi$, we have 
\begin{equation}
\Cpim f \leq S_{\pi} \left(\dfrac{m}{\tau}\right)\,. \label{eq:to_prove}
\end{equation}
We prove this inequality row by row. 
Let $q$ be a transition.
We distinguish the following cases:
\begin{itemize}
  \item if $q \in \Qsync$, then $(C^- f)_{p_{\pi}(q)} = (\Cpim f)_q$ for any $\pi$ (for any choice of an upstream place of $q$) 
  so that~\eqref{eq:to_prove} follows from~\eqref{eq:stationary_ineq_mp}.
  \item if $q$ has a unique upstream place $p$, with $p \in \Pconflict$, then
  for any $\pi$, $p_{\pi}(q) = p_{\pi^*}(q)$ so that~\eqref{eq:to_prove} follows from~\eqref{eq:right_annuler_mp}.
  \item assume now that $q_+$ is the priority transition of a place $p$ subject to priority.
  Then, by~\eqref{eq:stationary_ineq_mp}, 
  $m_p/\tau_p \geq a^-_{q_+p} f_{q_+} + a^-_{q_-p} f_{q_-} \geq a^-_{q_+p} f_{q_+}$ and 
  for $r \in q_+\inc \setminus \{p\}$, ${m_r}/{\tau_r} \geq a^-_{{q_+}r} f_{q_+}$.
  Finally, for any $r \in q_+\inc$, $a^-_{q_+p} f_{q_+} \leq {m_r}/{\tau_r}$.
  This proves~\eqref{eq:to_prove}.
  \item let $q_-$ be the non priority transition of a place $p$ subject to priority.
  Then $(C^- f)_{p_{\pi}(q_-)} = (\Cpim f)_{q_-}$ for any policy $\pi$,
  so that~\eqref{eq:to_prove} follows from~\eqref{eq:stationary_ineq_mp}. \qedhere
\end{itemize}
\end{proof}

As a consequence of \Cref{th:1}, we obtain a correspondence between the
stationary solutions of the continuous dynamics and the stationary solutions of the discrete dynamics. 
In order to highlight the parallel between the discrete and the continuous setting, we denote by $f_p$ the processing flow $m_p/\tau_p$ for every place~$p$.

\begin{corollary}\label{coro:corresp}
  \begin{enumerate}[(i)]
    \item Suppose $(m,w,f)$ defines a stationary solution of the continuous dynamics.
    Then, for the initial marking $M_p^0= m_p$, 
    setting $\rho:=f$, $u_p:=M_p^0$, and $u_q:=0$ yields a stationary solution
    of the discrete dynamics. 
    \item Conversely, suppose $(\rho,u)$ is a stationary
    solution of the discrete dynamics. Then, defining $f:=\rho$,
    setting $m_p:=\rho_p\tau_p$ for every place $p$,
    and defining $w$ according to~\eqref{eq:wpcns} and \eqref{eq:wp_annuler}
    yields a stationary solution of the continuous dynamics.
    \end{enumerate}
\end{corollary}

\begin{proof}
Both statements are straightforward.
We point out that~\eqref{eq:fixpoint_germ1_rho} reads $\rho_p = C^+ \rho_q$ 
and that~\eqref{eq:fixpoint_germ2_rho}--\eqref{eq:fixpoint_germ5_rho} are equivalent to $\rho_q = \min_{\pi} (\Cpim)^{-1} S_{\pi} \rho_p$.
The same relationship between the $f_q$ and the $f_p$ was established in the proof of Theorem~\ref{th:1}.
\end{proof}

An important
problem is to relate the stationary flow to the initial marking.
On top of the relations given by the invariants of the Petri nets, most
results in this direction are limited to nets without priorities,
as they rely on monotoni\-city properties of the dynamics. The next theorem identifies, however, a somehow special situation in which such a relation persists
even in the presence of priority. This
applies in particular to the Petri net of the next section.

\begin{theorem}
If a trajectory of the continuous Petri net converges towards a
stationary solution $(m^{\infty},w^{\infty},f^{\infty})$, 
if for this trajectory, there exists a policy $\pi$ that reaches the infimum in~\eqref{eq:matrix_dyn_f} at any time,
and if $0$ is a semi-simple eigenvalue of $(S_{\pi} C^+ - \Cpim)$ associated with this policy, 
then $f^{\infty}$ is uniquely determined by the initial marking.
\end{theorem}
(Recall that the eigenvalue $\lambda$ of a matrix $B$ is said to be {\em semi-simple} if the dimension of its eigenspace is equal to its algebraic multiplicity, that is, to the multiplicity of $\lambda$ as the root of the characteristic polynomial of $B$. In particular, if $0$ is a semi-simple eigenvalue of $B$, then the kernel of $B$ and its range space are complementary subspaces.)

\begin{proof}
Under the conditions of the theorem, there exists a policy $\pi$ such that, for any $t$,
\begin{align}
  S_{\pi} \dot{m}(t) &= (S_{\pi}C^+ - \Cpim ) f(t) \,,  \label{eq:eq1} \\
  S_{\pi} m(t) &= D_{\pi} \Cpim f(t) \,,
\end{align}
as shown in the proof of Proposition~\ref{thm:fc_prio_unicity}.

Since $0$ is a semi-simple eigenvalue of $(S_{\pi}C^+ - \Cpim)$, the same property holds for the matrix $(S_{\pi}C^+ - \Cpim) (D_\pi \Cpim)^{-1} =  (S_{\pi}C^+ (\Cpim)^{-1} - I)D_{\pi}^{-1}$. Therefore, the kernel of this matrix and its range space are complementary subspaces. We denote by $Q$ the projection onto the former along the latter.

By~\eqref{eq:eq1}, we obtain that $Q S_{\pi} \dot{m}(t) = Q (S_{\pi}C^+ - \Cpim ) f(t) = 0$, so that $Q S_{\pi} {m}(t)$ is independent of time, and 
\[
  Q S_{\pi} {m}(0) = Q S_{\pi}m_{\infty} = Q D_{\pi} \Cpim f_{\infty} \,.
\]
Moreover, as $(m^{\infty},w^{\infty},f^{\infty})$ is a stationary solution of the continuous dynamics, Equation~\eqref{eq:right_annuler} holds and 
$D_{\pi} \Cpim f_{\infty}$ belongs to the kernel of $(S_{\pi}C^+ (\Cpim)^{-1} - I)D_{\pi}^{-1}$.
Therefore,
\[
  f_{\infty} = (\Cpim)^{-1} D_{\pi}^{-1} Q S_{\pi} {m}(0)\,. \qedhere
\]
\end{proof}

\section{Experimental results}
\label{sec:numerical experiments}
In this section, we illustrate our results on  the model
of an emergency call center with two treatment levels, introduced in~\cite{allamigeon2015performance}. In this simplified model of an emergency call center, emergency calls are handled by a first level of operators who dispatch them into three categories: extremely urgent, urgent and non urgent. Non-urgent calls (proportion $\mu_4$ of the calls) are entirely processed by level~1 operators. Extremely urgent ($\mu_2$) and urgent calls ($\mu_3$) are transfered to level~2 operators. Extremely urgent calls have priority over urgent calls (but cannot interrupt a talk between an operator of level~2 and an urgent call).

This emergency call center can be modeled by a Petri net with free choice and priority routing, as depicted in Figure~\ref{fig:call_center}. Place $p_3$ is a conflict place with a fluid stationary routing, with proportions $\mu_2$, $\mu_3$, $\mu_4$, representing the dispatching of calls
into the categories ``extremely urgent'', ``urgent'' and ``non urgent'' respectively.
Every arc has a valuation equal to one.
The initial marking $M_1^0$ (resp.\ $M_2^0$) of place $p_1$ ($p_2$) denotes the available number of operators of level~1 (level~2) in the call center.

It was observed in~\cite{allamigeon2015performance} that the discrete dynamics
has a pathological feature: when certain arithmetic relations between
the time delays are satisfied, the discrete time trajectory
may not converge to a stationary solution, and
its asymptotic throughput may differ from the throughput
of the stationary solution. It follows from our correspondence
result (Corollary~\ref{coro:corresp}) that the continuous dynamics
has the same stationary solutions. We shall observe that,
in this continuous setting, the trajectory converges towards a stationary solution,
so that the former pathology vanishes.

To compute the (fluid approximation) of the discrete dynamics, simulations have been performed in exact (rational)
arithmetics, using the GMP library~\cite{gmplib}. The throughput of
transitions $q_5$ and $q_6$ (see Figure~\ref{fig:call_center}), for the  discrete dynamics,
are compared in~\Cref{fig:phasesDiagramm} to the throughputs of the stationary solutions, computed by Equations~\eqref{eq:stat_nodes_law} and~\eqref{eq:right_annuler}.

\begin{figure}[t]
\begin{center}
\includegraphics{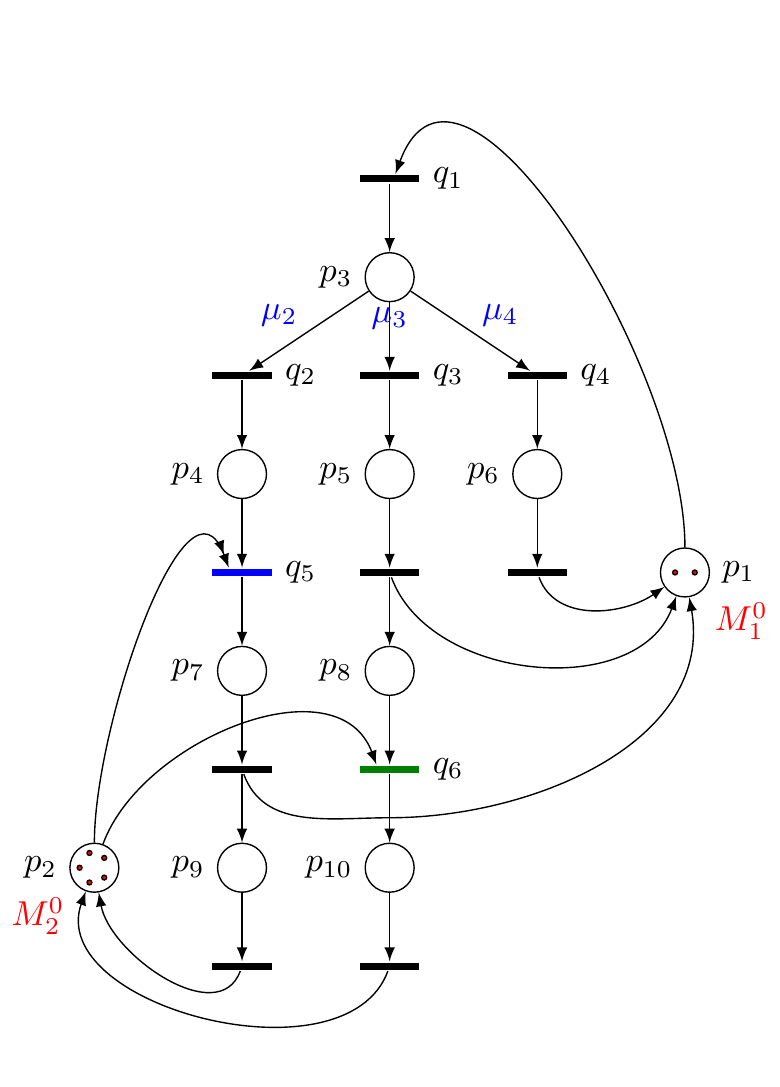}
\end{center}
\vspace{-2em}
\caption{Petri net of a simplified emergency call center. Place $p_2$ is subject to priority routing. The initial markings of the places different from $p_1$ and $p_2$ are null. The holding times $(\tau_1,\tau_2,\tau_3,\tau_4,\tau_5,\tau_6,\tau_7,\tau_8,\tau_9,\tau_{10})$ are the following: $(0.01, 0.01, 0.01, 4, 3, 3, 1, 0.01, 6, 7)$.} \label{fig:call_center}
\end{figure}

The dynamics expressed by~\eqref{eq:mp}--\eqref{eq:wp}, \eqref{eq:sync} and~\eqref{eq:conflict_routing}--\eqref{eq:p2} belongs to the class of hybrid automata \cite{henzinger2000theory}, which can handle piecewise linear but discontinuous dynamics like ours. We simulate our dynamics with the tool
SpaceEx~\cite{frehse2011spaceex}, which is a verification platform for hybrid systems. The particularity of SpaceEx is that it computes a sound over-approximation of the trajectories.

At the scale of Figure~\ref{fig:phasesDiagramm}, the lower and upper bounds to the values of the throughputs, computed by SpaceEx, coincide with the shape of the stationary throughputs curve.
Table~\ref{tab:1} compares the numerical values of these lower and upper bounds to the stationary throughputs for a few values of $M_2^0/M_1^0$. We observe that the over-approximation computed by SpaceEx provides an accurate estimate of the 
stationary throughput computed via Equations~\eqref{eq:stat_nodes_law}--\eqref{eq:right_annuler}.
This tends to show that the continuous dynamics converges towards the stationary throughputs, unlike the discrete dynamics.

Note that the experiments made with SpaceEx did not terminate for $M_2^0/M_1^0 = 0.6$: this seems to be related with the larger number of switches between the states of the automaton at this frontier between two different phases.

\begin{figure}[t]
\begin{center}
\includegraphics{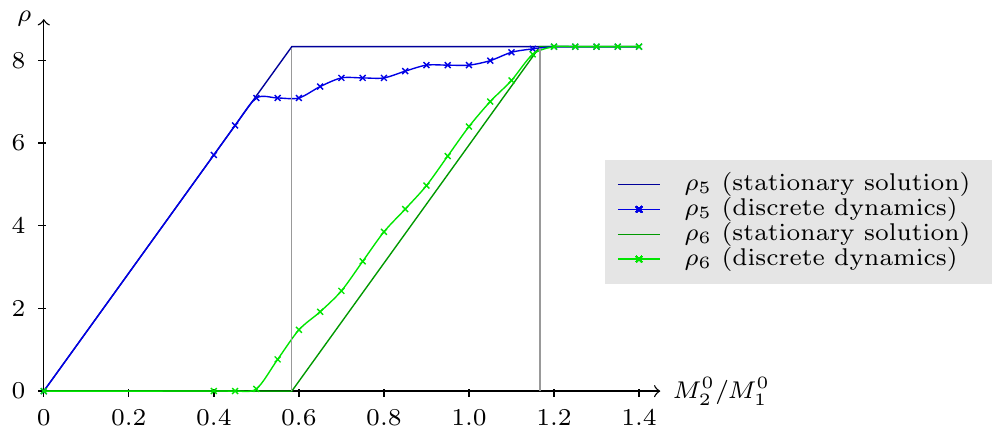}
\end{center}
\caption{Comparison of the throughputs of the discrete dynamics simulations with the theoretical throughputs (fluid model).
} \label{fig:phasesDiagramm} 
\end{figure}

\begin{table} 
\begin{center}
\begin{tabular}{r || r | r | r | r | r | r}
\toprule
$M_2^0/M_1^0$ & 0.2 & 0.4 & 0.6 & 0.8 & 1.0 & 1.2\\
\midrule
$\rho_5$ & 2.857 & 5.714 & 8.333 & 8.333 & 8.333 & 8.333 \\
$f_5^{\text{up}}$ & 2.865 & 5.716 & --- & 8.338 & 8.339 & 8.340 \\
$f_5^{\text{down}}$ & 2.849 & 5.707 & --- & 8.328 & 8.328 & 8.327 \\
\midrule
$\rho_6$ & 0 & 0 & 0.238 & 3.095 & 5.952 & 8.333 \\
$f_6^{\text{up}}$ & $< 0.001$ & $< 0.001$ & --- & 3.107 & 5.968 & 8.340 \\
$f_6^{\text{down}}$ & 0 & 0 & --- & 3.083 & 5.936 & 8.327 \\
\bottomrule
\end{tabular}
\end{center}
\caption{Lower and upper bounds of the throughputs of the continuous dynamics computed by SpaceEx, and comparison to the stationary throughputs} \label{tab:1}
\end{table}

\section{Conclusion}
We introduced a hybrid dynamical system model for continuous
Petri nets having both free choice and priority places,
and showed that there is a correspondence between
the stationary solutions of the continuous dynamics and 
the discrete one. An advantage of the continuous
setting is that some pathologies of
the discrete model (failure of convergence to a stationary
solution) may vanish. This is the case in particular
on a case study (emergency call center). We
leave it for further work to see under which generality
the convergence to the stationary solution can be established.

\subsection*{Acknowledgments}
An abridged version of the present work appeared in the Proceedings of
the conference Valuetools 2016. We thank the referees
for their detailed comments and for pointing out relevant references. 

\newcommand{\etalchar}[1]{$^{#1}$}

\end{document}